\documentclass[
final
, nomarks
]{dmtcs-episciences}

\usepackage[utf8]{inputenc}
\usepackage{subfigure}

\newtheorem{THEOREM}{Theorem}
\newenvironment{theorem}{\begin{THEOREM} \hspace{-.85em} {\bf :} \rm}                        {\end{THEOREM}}
\newtheorem{LEMMA}[THEOREM]{Lemma}
\newenvironment{lemma}{\begin{LEMMA} \hspace{-.85em} {\bf :} \rm}                      {\end{LEMMA}}
\newtheorem{COROLLARY}[THEOREM]{Corollary}

\newenvironment{proof}{\noindent {\bf Proof:} \hspace{.677em}}                      {\qed }
\newtheorem{DEFINITION}{Definition}
\newenvironment{definition}{\begin{DEFINITION} \hspace{-.85em} {\bf :} \rm}
                            {\end{DEFINITION}}
\newtheorem{CLAIM}[THEOREM]{Claim}
                      {\end{CLAIM}}
\newtheorem{PROPOSITION}[THEOREM]{Proposition}
                      {\end{PROPOSITION}}
\input{epsf}
\author{Yen Hung Chen\thanks{Email:yhchen@utaipei.edu.tw}
	}
\title[Approximability results for the $p$-centdian and the converse centdian problems]{Approximability results for the $p$-centdian and the converse centdian problems\thanks{An extended abstract of this paper was presented in the  18th International Conference Information and Knowledge Engineering (IKE 2019). This work was supported in part by the Ministry of Science and Technology of Taiwan under Contract MOST 105-2221-E-845-002, 107-2115-M-845-001 and 110-2221-E-845-001.}}
\affiliation{
  Department of Computer Science, University of Taipei, No.1, Ai-Guo West Road, Taipei, Taiwan
}
\received{2020-11-03}
\revised{2022-07-11}
\accepted{2022-10-15}
\begin{document}
\publicationdetails{24}{2022}{2}{8}{6877}
\maketitle
\begin{abstract}
 Given an undirected graph $G=(V,E)$ with a nonnegative edge length function and  an integer $p$, $0 < p < |V|$, the $p$-centdian problem is to find $p$ vertices (called the {\it centdian set}) of $V$ such that the {\it eccentricity} plus {\it median-distance} is minimized, in which 
the {\it eccentricity} is the maximum (length) distance of all vertices to their nearest {\it centdian set} and the {\it median-distance} 
is the total (length) distance of all vertices to their nearest {\it centdian set}. The {\it eccentricity} plus {\it median-distance} is called the {\it centdian-distance}.  The purpose of the $p$-centdian problem is to find $p$ open facilities (servers) which satisfy the quality-of-service of  the  minimum total distance ({\it median-distance}) and the  maximum distance ({\it eccentricity}) to their service customers,  simultaneously. If we converse the two criteria, that is given the bound of the {\it centdian-distance} and the objective function is to minimize the cardinality of the {\it centdian set}, this problem is called  the converse centdian problem. 
In this paper, we prove the $p$-centdian problem is NP-Complete. Then we design the first non-trivial brute force exact algorithms for the $p$-centdian problem  and the converse centdian problem, respectively.  Finally, we design two approximation algorithms for both problems.
\end{abstract}

\section{Introduction}
The {\it $p$-center problem}~\cite{Hak,kar,Tan} and {\it $p$-median problem}~\cite{Hak,kar1,Tan} are fundamental problems in graph theory and operations research.   
Let $G=(V,E,\ell)$ be an   undirected  graph  with  $\ell: E \rightarrow R^+$ on the edges. Given a vertex set $V^{\prime} \subset V$, for each vertex $v\in V$, we let  $d(v,V^{\prime})$ denote the shortest distance from $v$ to $V^{\prime}$  
(i.e., $d(v,V^{\prime})=\min_{u \in V^{\prime}} d(u,v)$, in which $d(u,v)$ is the length of the {shortest} path of $G$ from $u$  to $v$). 
The {\it eccentricity} of a vertex set $V^{\prime}$ is defined as the maximum distance of $d(v,V^{\prime})$ for all $v \in V$, denoted by $\pounds_C(V^{\prime})$ (i.e., $\pounds_C(V^{\prime})=max_{v \in V} d(v,V^{\prime})$).  The {\it median-distance} $\pounds_M(V^{\prime})$ of $V^{\prime}$ denotes the total distance of $d(v,V^{\prime})$ for all $v$ in $V$ 
(i.e., $\pounds_M(V^{\prime})=\sum_{v\in V} d(v,V^{\prime})$). 
Given an  undirected complete graph $G=(V,E,\ell)$ with a nonnegative edge length  function $\ell$ and an integer $p$, $0<p<|V|$,  the {\it $p$-center problem} ($p$CP) 
(respectively, the {\it $p$-median problem} ($p$MP)) is to find a vertex set $V^{\prime}$ in $V$, $|V^{\prime}|=p$, such that the {\it eccentricity} 
(respectively, the {\it median-distance}) of $V^{\prime}$ is minimized~\cite{Hak,kar,kar1,Tan}. 
Both problems had been shown to be NP-Complete~\cite{gar,kar,kar1}. Hence, many approximation algorithms~\cite{Ary,gon,gup,hoc,ple,Shmoys} and inapproximability results~\cite{hoc2,hsu,jain}  had  been  proposed for both problems.  
These two problems have many applications in  the network location, clustering, and social networks~\cite{abu,cha,das1,das2,dre,Hak,hoc,kar,kar1,mih,Pacheco,rev,rev2,tak,Tam,Tan}.  

Given  a set of customers on the network, the network location theory is concerned with the optimal locations of new facilities (servers) to minimize transportation distances (costs)  of serving these customers and  consider the population density area. 
The most fundamental problems of the network location theory are  the $p$CP and the $p$MP, respectively.
 The $p$CP is suitable for emergency services where the objective is to have the farthest  customers as close as possible to their facility centers. But this solution of the $p$CP may cause a  substantial increase in total distance (cost), thus this result takes a huge loss of the spatial efficiency.  
The $p$MP is suitable for locating  facilities providing a routine service, by  minimizing the average distances from customers to these selected facilities.  The solution of the $p$MP is beneficial in serving
centrally located and  high-population density areas but  sacrifices the remote and low-population  density areas~\cite{Perez1,Perez2,Tam2}. Motivated by the application of finding {\it p} open facilities (servers) which satisfy the quality-of-service of the  minimum total distance ({\it median-distance}) 
and the maximum distance  ({\it eccentricity}) 
to their service customers, simultaneously~\cite{Hal1,Hal2,hooker,Perez1,Perez2,Tam2},  
Halpern~\cite{Hal1,Hal2}  introduced a convex combination of the $1$CP and the $1$MP,  which he called the  {\it $1$-centdian problem}. Hooker et al.~\cite{hooker} studied the generalization of the {\it $1$-centdian problem}, called  the {\it $p$-centdian problem}.  Given an  undirected complete graph $G=(V,E,\ell)$ with a nonnegative edge length  function $\ell$, a real number $\lambda$, $0 \le \lambda \le 1$, and an integer $p$, $0<p<|V|$, the {\it $p$-centdian problem} ($p$DP) is to find a vertex set $V^{\prime}$ in $V$, $|V^{\prime}|=p$, such that the 
$\lambda\pounds_C(V^{\prime})+(1-\lambda)\pounds_M(V^{\prime})$ is  minimized~\cite{hooker}.  The vertex set $V^{\prime}$ is called  the {\it centdian set} 
and $\lambda\pounds_C(V^{\prime})+(1-\lambda)\pounds_M(V^{\prime})$ is called the {\it centdian-distance}.  
If the {\it centdian set} can be the continuum set of points on the edges of $G$, Hooker et al.~\cite{hooker} proposed the possible {\it centdian set} for the $p$DP. 
Perez-Brito et al.~\cite{Perez1} fixed the flaw of Hooker et al.~\cite{hooker} theorem for the $p$DP.  
Tamir et al.~\cite{Tam2} presented a polynomial time exact algorithm for the $p$DP on trees. 
Ben-Moshe et al.~\cite{BenMoshe} gave  $O(|V|log |V|)$ time exact algorithms for the $1$DP  on cycle graphs and cactus graphs, respectively. If  the induced subgraph by the {\it centdian set} is connected,  Nguyen et al.~\cite{Nguyen} proposed a  linear time algorithm for the $p$DP on unweighted block graphs  and proved the problem is NP-Complete on weighted block graphs. If $\lambda=0$, the $p$DP is equal to the  $p$MP, and however $\lambda=1$ the $p$DP is equal to the $p$CP. Hence, it is not hard to see that the $p$DP is NP-hard.  However, it is still unclear whether there exists a polynomial time deterministic  approximation algorithm for the $p$DP. Given an undirected graph $G=(V,E)$ and two independent minimization 
criteria with a bound on the first criterion, a {\it generic bicriteria network design problem} involves 
the minimization of  the second criterion but satisfies the bound on the first criterion  
 among all  possible subgraphs from $G$~\cite{Marathe}. 
Many multiple criteria problems had been studied~\cite{Chow,Goel,Kadaba,Kompella,Marathe}. 
Clearly, the $p$CP, $p$MP, and $p$DP  are one kind of  {\it bicriteria network design problems}. 
The first criterion is the cardinality of the vertex set $V^{\prime}$ and the second is the {\it eccentricity}, {\it median-distance}, and {\it centdian-distance}, respectively.  
Hence,  if we converse the two criteria, that is given the bound of the  {\it eccentricity} 
from each vertex to $V^{\prime}$ and the objective function  is to minimize the cardinality of the $V^{\prime}$, this problem is called   the {\it converse $p$-center problem} (also called the {\it balanced $p$-center problem})~\cite{bar1,bar2,gar}. 
Given a graph $G=(V,E,\ell)$ with a nonnegative edge length  function $\ell$ and an integer $U$, $U>0$,  
the {\it converse $p$-center problem}  is to find  a vertex set $V^{\prime}$ in $V$ with minimum cardinality such that  the {\it eccentricity} of $V^{\prime}$ is at most  $U$~\cite{bar1,bar2,gar}. This problem had been shown to be NP-Complete~\cite{gar} and  
a $(\log U+1)$-approximation algorithm had been proposed~\cite{bar1,bar2}.
However, the converse version of the {\it $p$-centdian problem} is undefined. 
Hence, we present the converse version of the {\it $p$-centdian problem},  called  the {\it converse centdian problem}. 
Given a  graph $G=(V,E,\ell)$ with a nonnegative edge length  function $\ell$ and two integers $\lambda$ and $U$, $0 \le \lambda \le 1$, $U>0$,  
the {\it converse centdian problem} (CDP) is to find  a vertex set $V^{\prime}$ in $V$ with minimum cardinality such that  $\lambda\pounds_C(V^{\prime})+(1-\lambda)\pounds_M(V^{\prime})$ of $V^{\prime}$ is at most  $U$. 
In this paper, we focus on a  special case of the {\it centdian-distance} for the $p$DP (respectively, CDP) : 
$\pounds_C(V^{\prime})+\pounds_M(V^{\prime})$ and discuss the complexity, the non-trivial brute force exact algorithms, and  the approximation algorithms for 
the  $p$DP and CDP, respectively.
First, we  prove that the $p$DP is  NP-Complete even when the {\it centdian-distance} is $\pounds_C(V^{\prime})+\pounds_M(V^{\prime})$.   
Then we present the first non-trivial brute force exact algorithms for the $p$DP and CDP, respectively. Finally, we
design a  $(1+\epsilon)$-approximation algorithm for the $p$DP  satisfying the cardinality of the {\it centdian set} is less than or equal to $(1+1/\epsilon)(ln|V|+1)p$ and a
$(1+1/\epsilon)(ln|V|+1)$-approximation algorithm for the CDP satisfying the {\it centdian-distance} is less than or equal to $(1+\epsilon)U$,  in which $\epsilon>0$, respectively.

The rest of this paper is organized as follows. In Section~\ref{pre}, some 
definitions and notations are given. In Section~\ref{hardnessresult}, we prove that the $p$DP is  NP-Complete even when the {\it centdian-distance} is $\pounds_C(V^{\prime})+\pounds_M(V^{\prime})$. In Section~\ref{sectexact}, we present   non-trivial brute force exact algorithms for the $p$DP and CDP, respectively. 
In Section~\ref{sectapp}, we design  a  $(1+\epsilon)$-approximation algorithm for the $p$DP 
satisfying the cardinality of the {\it centdian set} is less than or equal to $(1+1/\epsilon)(ln|V|+1)p$.
In Section~\ref{sectapp2}, we design  a $(1+1/\epsilon)(ln|V|+1)$-approximation algorithm for the CDP
satisfying the {\it centdian-distance} is less than or equal to $(1+\epsilon)U$,  in which $\epsilon>0$. 
Finally, we make a conclusion in Section~\ref{result}.  

\section{Preliminaries}
 \label{pre}
In this paper, a graph is simple, connected and undirected.
By $G=(V,E,\ell)$, we denote a graph $G$ with vertex set $V$, edge set $E$, and edge length function $\ell$. The edge length function is assumed to be nonnegative.  
We  use $|V|$ to denote the cardinality of vertex set $V$. Let $(v,v^{\prime})$ denote an edge connecting two vertices $v$ and $v^{\prime}$. 
For any vertex $v \in V$ is said to be $adjacent$ to a vertex $v^{\prime} \in V$ if vertices  $v$ and $v^{\prime}$  share a common edge $(v,v^{\prime})$.   

\begin{definition}  For $u,v \in V$, $SP(u,v)$ denotes a shortest path
between $u$ and $v$ on $G$. The shortest path length is denoted
by $d(u,v)=\sum_{e \in SP(u,v)}\ell (e)$.
\end{definition}

\begin{definition} Let $H$ be a vertex set  of $V$. For a vertex $v \in V$,
we let $d(v,H)$  denote the shortest distance from $v$ to $H$,
i.e., $d(v,H)= \min_{h \in H} \{d(v,h)\}$.
\end{definition}

\begin{definition} Let $H$ be a vertex set  of $V$. 
The {\it eccentricity} of $H$, denoted by
$\pounds_C(H)$, is the maximum distance of $d(v,H)$ for  all $v \in V$, i.e., $\pounds_C(H)=max_{v \in V} d(v,H)$. 
\end{definition}

\begin{definition} Let $H$ be a vertex set  of $V$. 
The {\it median-distance} of $H$, denoted by
$\pounds_M(H)$, is the  the total distance of $d(v,H)$ for  all $v \in V$, i.e.,  $\pounds_M(H)=\sum_{v\in V} d(v,H)$. 
\end{definition}

\begin{description}  
\item[\textbf{$p$CP}] ($p$-center problem)~\cite{Hak,kar,Tan}
\item[Instance:] A  connected, undirected, complete graph $G=(V,E,\ell)$ and an integer $p>0$.
\item[Question:] Find a  vertex set $V^{\prime}$, $|V^{\prime}|=p$, such that the {\it eccentricity} of $V^{\prime}$ is minimized.
\end{description}

\begin{description} 
\item[\textbf{$p$MP}] ($p$-median problem)~\cite{Hak,kar1,Tan}
\item[Instance:] A  connected, undirected, complete  graph $G=(V,E,\ell)$ and an integer $p>0$.
\item[Question:] Find a  vertex set $V^{\prime}$, $|V^{\prime}|=p$, such that the {\it median-distance} of $V^{\prime}$ is minimized.
\end{description}

\begin{description} 
\item[\textbf{$p$DP}] ($p$-centdian problem)~\cite{hooker}
\item[Instance:] A   connected, undirected, complete graph $G=(V,E,\ell)$ and an integer $p>0$.
\item[Question:] Find a  vertex set $V^{\prime}$, $|V^{\prime}|=p$, such that $\pounds_C(V^{\prime})+\pounds_M(V^{\prime})$ of $V^{\prime}$ is minimized.
\end{description}

For the $p$DP, we have two criteria. The first criterion is the cardinality of the {vertex set}  $V^{\prime}$ and the second is the $\pounds_C(V^{\prime})+\pounds_M(V^{\prime})$.  The vertex set $V^{\prime}$ is called  the {\it centdian set} 
and $\pounds_C(V^{\prime})+\pounds_M(V^{\prime})$ is called  the {\it centdian-distance}. 
Hence, we can converse the two criteria, that is given the bound of the {\it centdian-distance}  of  the {\it centdian set}   and 
the objective function  is to minimize the cardinality of the {\it centdian set}.

\begin{description}   
\item[\textbf{CDP}] (converse centdian problem)
\item[Instance:] A connected, undirected graph $G=(V,E,\ell)$ and an integer $U>0$.
\item[Question:] Find a  vertex set $V^{\prime}$ with $\pounds_C(V^{\prime})+\pounds_M(V^{\prime}) \leq U$ such that  the cardinality  of the  
$V^{\prime}$ is minimized.
\end{description}
 The following examples illustrate the $p$DP and the CDP.  Consider the instance shown in Fig.~\ref{Fig1}, in which the graph $G=(V,E,\ell)$ 
and integers  $p=2$ and $U=117$. 
An optimal solution of $G$ for the $p$DP is shown in Fig.~\ref{Fig2}, in which the {centdian-distance} is $252$
An optimal solution of $G$ for the CDP is shown in Fig.~\ref{Fig3}, in which the {centdian set} is $\{A,B,D\}$.

\begin{figure}
\begin{center}
\centerline{\epsfbox{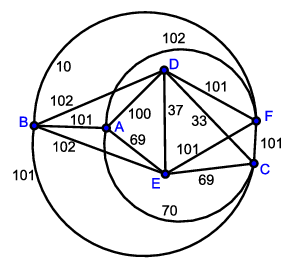}}
\caption{An instance: complete graph $G=(V,E,\ell)$, $p=2$ and $U=117$.}
\label{Fig1}
\end{center}
\end{figure}

\begin{figure}
\centerline{\epsfbox{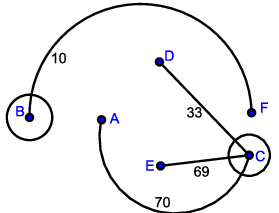}}
\caption{The optimal solution $\{B,C\}$ for the $2$DP. (Note that $\pounds_C(\{B,C\})+\pounds_M(\{B,C\})=252$)}
\label{Fig2}
\end{figure}

\begin{figure}
\centerline{\epsfbox{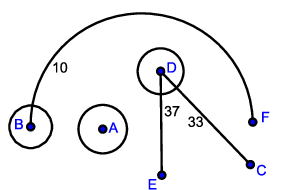}}
\caption{The optimal solution $\{A,B,D\}$ for the CDP. (Note that $\pounds_C(\{A,B,D\})+\pounds_M(\{A,B,D\})=117$)}
\label{Fig3}
\end{figure}

In this paper, 
we will prove that the $p$DP  is NP-Complete  by 
a {\it reduction}  from the   {\it dominating set problem}~\cite{ber,coc,ore,west} to the $p$DP. 
Hence, we  review the definition of the {\it dominating set problem}.  
A {\it dominating set} of $G$, denoted by  $\mathcal {Z}$, is a subset of $V$  such that each vertex in $V \setminus \mathcal {Z}$ is {\it adjacent} to a vertex in 
$\mathcal {Z}$~\cite{ber,coc,ore,west}. 
\begin{description}  
\item[\textbf{DSP}] (dominating set problem)~\cite{ber,coc,ore,west}
\item[Instance:] A  connected, undirected graph $G=(V,E)$.
\item[Question:] Find a {\it  dominating set} $\mathcal {Z^{\prime}}$ with minimum cardinality.
\end{description}

Note that the DSP had been shown to be NP-Complete~\cite{gar}.
 
Since  our approximation algorithm for the $p$DP is based on  the {\it set cover problem}~\cite{chv,joh,lov}.  We  also review the definition of the {\it set cover problem}.  Given a finite set $\mathcal{U}$ of elements and a collection 
$\mathcal{S}$ of (non-empty) subsets of $\mathcal{U}$.   
A  {\it set cover}~\cite{chv,joh,lov} is to find  a subset $\mathcal{S}^{\prime} \subseteq \mathcal{S}$ such that every element in  $\mathcal{U}$ belongs to at least one element of $\mathcal{S}^{\prime}$.

\begin{description} 
\item[\textbf{SCP}] (Set cover problem)~\cite{chv,joh,lov}
\item[Instance:]  A finite set $\mathcal{U}$ of elements, a collection $\mathcal{S}$  of (non-empty) subsets of $\mathcal{U}$.
\item[Question:] Find a {\it set cover} $\mathcal{S}^{\prime\prime}$  such that the number of sets in $\mathcal{S}^{\prime\prime}$  is minimized. 
\end{description}

\section{Hardness Result for the $p$DP}
\label{hardnessresult}
In this section, we prove that the $p$DP is NP-Complete. We transform the DSP to the $p$DP  by the {\it reduction}. Hence we need to define 
$p$DP and DSP decision problems.

\begin{description}
\item[\textbf{$p$DP Decision Problem}]
\item[Instance:] A  connected, undirected complete  graph $G=(V,E,\ell)$ and two integers $p>0$ and $U>0$.
\item[Question:]  Does there exist a  vertex set $V^{\prime}$, $|V^{\prime}|=p$, such that $\pounds_C(V^{\prime})+\pounds_M(V^{\prime}) \leq U$?  
\end{description}

\begin{description}
\item[\textbf{DSP Decision Problem}]
\item[Instance:] A  connected, undirected graph $G=(V,E)$, and a positive integer $\kappa$.
\item[Question:] Does there exist a  dominated set $\mathcal {Z}$ such that 
$|\mathcal {Z}|$ is less than or equal to $\kappa$?  
\end{description}

\begin{theorem} The $p$DP decision problem is  NP-Complete.  \label{NPC}
\end{theorem}
\begin{proof}
First, it is easy to see that  the  $p$DP decision problem  is in NP. Then we show the {\it reduction}: the transformation from 
the  DSP decision problem  to the   $p$DP decision problem.

Let a graph $G=(V,E)$ and  a positive integer $\kappa$ be an instance of the DSP decision problem. 
 We transform it into an instance of the $p$DP decision problem, say $\overline{G}=(\overline{V},\overline{E}, \ell)$ 
and two positive integers $p$ and $U$, as follows. 

\begin{description}
\item[] $\overline{V}={V}$. 
\item[] $\overline{E}={E}$. 
\item[] For each edge $(u,v) \in \overline{E}$,  
\begin{equation}
 \ell(u,v)=  \left\{\begin{array}{ll}
                     1,& \mbox{if $(u,v) \in E$}\\
                     d(u,v), & \mbox{otherwise.}
                \end{array} \right. 
\end{equation} 
\item[] $U=|V|-\kappa+1$ and $p=\kappa$.
\end{description}

Now, we show that  there is  a dominating set $\mathcal{Z}$ such that $|\mathcal {Z}|$ is  $\kappa$ if and only if there is 
a  vertex set $\overline{V^{\prime}}$ in  $\overline{G}$  such that the $|\overline{V^{\prime}}|$ is $p$ and 
 $\pounds_C(\overline{V^{\prime}})+\pounds_M(\overline{V^{\prime}})$ is $U$. 

(Only if) If there exists a dominating set $\mathcal{Z}$ in $G$ and the cardinality of $\mathcal{Z}$ is  at most $\kappa$. Then we 
 choice the corresponding vertex set $\overline{V^{\prime}}$ in $\overline{G}$ of the dominating set $\mathcal{Z}$ in $G$. Hence, we have $\pounds_C(\overline{V^{\prime}})=1$ and $\pounds_M(\overline{V^{\prime}})=|V|-\kappa$.
(If) If there exists a vertex set $\overline{V^{\prime}}$ in  $\overline{G}$  such that   $|\overline{V^{\prime}}|$ is $p$ and 
 $\pounds_C(\overline{V^{\prime}})+\pounds_M(\overline{V^{\prime}})$ is $U$. Clearly, each vertex $v$ in $V \setminus \overline{V^{\prime}}$, 
$d(v,\overline{V^{\prime}})=1$, otherwise  $\pounds_C(\overline{V^{\prime}})+\pounds_M(\overline{V^{\prime}})>U=|V|-p+1$. Hence, 
we  choice the corresponding vertex set $\mathcal{Z}$ in $G$ of  the vertex set  $\overline{V^{\prime}}$ in $\overline{G}$ 
and $\mathcal{Z}$ is  a dominating set in $G$ with $|\mathcal{Z}|=p$.
\end{proof}

\section{Exact Algorithms for the $p$DP and CDP} \label{sectexact}
In this section, we show  integer programmings to solve the $p$DP and CDP, respectively. We combine the integer programmings for the $p$MP and $p$CP by~\cite{das1}.
Given an undirected complete graph $G=(V,E,\ell)$ with a nonnegative edge length  function $\ell$,  the $p$DP can be formulated as an integer programming  
($I$) as follows. 

\begin{description}
 \item[] 
\begin{eqnarray}
 &\mbox{minimize}& \sum_{i \in V} \sum_{j \in V} d(i,j) x_{i,j}+C \\
&\mbox{subject to}& \nonumber \\ 
        &&\sum_{j \in V} x_{i,j} = 1, \forall i \in V \\
			   &&\sum_{j \in V} y_j = p \\
				  &&x_{i,j} \le  y_j,  \forall i,j \in V \\
           	&&\sum_{j \in V} d(i,j)x_{i,j} \le C, \forall i \in V  \\
						&&x_{i,j},y_j \in \{0,1\} \\
						 &&C \ge 0, \\ \nonumber 
\end{eqnarray}
\end{description}
where the variable $y_j=1$ if and only if vertex $j$ is chosen as a
centdian, and the variable $x_{i,j} = 1$ if and only if $y_j = 1$ and vertex $i$ is assigned to vertex $j$, 
 and $C$ is a feasible {\it eccentricity}.
 For  completeness, we list the exact algorithm for the $p$DP  as follows.

\begin{description}  
\item[\textbf{Algorithm OPT-$p$DP}]
\item[\textbf{Input:}]   A  connected, undirected complete graph $G=(V,E,\ell)$ with  a nonnegative length function $\ell$ on edges and  an integer $p>0$.
\item[\textbf{Output:}]  A vertex set $P_{opt}$ with $|P_{opt}|=p$. 
\item[1.] Use the integer programming ($I$) to find all  $y_j=1$ and put the corresponding vertex $j$ of $y_j$ to $P_{opt}$. 
\item[2.] Return $P_{opt}$.
\end{description}

It is  easy to show that Algorithm OPT-$p$DP is an exact algorithm for the $p$DP. However, to solve an integer programming is NP-hard~\cite{cormen,gar}.
Hence, next section we show $(1+\epsilon)$-approximation algorithm for the $p$DP satisfying the cardinality of centdian set is less than or equal to $(1+1/\epsilon)(ln|V|+1)p$,  $\epsilon>0$. 

Next, we modify integer programming ($I$) to design another integer programming ($II$) for the CDP with an integer $U$ as follows.

\begin{description}
 \item[] 
\begin{eqnarray}
 &\mbox{minimize}& \sum_{j \in V} y_j \\
&\mbox{subject to}& \nonumber \\ 
        &&\sum_{j \in V} x_{i,j} = 1, \forall i \in V \\
			   &&  \sum_{i \in V} \sum_{j \in V} d(i,j) x_{i,j}+C \le U\\
				  &&x_{i,j} \le  y_j,  \forall i,j \in V \\
           	&&\sum_{j \in V} d(i,j)x_{i,j} \le C, \forall i \in V  \\
						&&x_{i,j},y_j \in \{0,1\} \\
						 &&C \ge 0. \\ \nonumber 
\end{eqnarray}
\end{description}
 
 For  completeness, we list the exact algorithm for the CDP  as follows.

\begin{description}  
\item[\textbf{Algorithm OPT-CDP}]
\item[\textbf{Input:}]   A  connected, undirected complete graph $G=(V,E,\ell)$ with  a nonnegative length function $\ell$ on edges and  an integer $U>0$.
\item[\textbf{Output:}]  A vertex set $P_{opt}$ with $\pounds_C(P_{opt})+\pounds_M(P_{opt}) \le U$. 
\item[1.] Use the integer programming ($II$) to find all  $y_j=1$ and put the corresponding vertex $j$ of $y_j$ to $P_{opt}$. 
\item[2.] Return $P_{opt}$.
\end{description}

\section{An Approximation Algorithm for the $p$DP}
\label{sectapp}
In this section, 
 we show $(1+\epsilon)$-approximation algorithm for the $p$DP satisfying the cardinality of {\it centdian set} is less than or equal to $(1+1/\epsilon)(ln|V|+1)p$, $\epsilon>0$. 
First, we relax the integer programming  ($I$) for the $p$DP to the linear programming ($I_L$) to solve the $p$DP called the fractional $p$DP
as follows.

\begin{description}
\item[] 
\begin{eqnarray}
&\mbox{minimize}& \sum_{i \in V} \sum_{j \in V} d(i,j) x_{i,j}+C \\
&\mbox{subject to}& \nonumber \\ 
   &     &\sum_{j \in V} x_{i,j} = 1, \forall i \in V \\
			   &&\sum_{j \in V} y_j = p \\
				  &&x_{i,j} \le  y_j,  \forall i,j \in V \\
           	&&\sum_{j \in V} d(i,j)x_{i,j} \le C, \forall i \in V 	 \\
						 &&0\le x_{i,j},y_j \le 1\\
						 &&C \ge 0. \\ \nonumber 
\end{eqnarray}
\end{description}
 The main difference between $I_L$ and $I$ is that $y_j$ and $x_{i,j}$ can take rational values between $0$ and $1$ for  $I_L$. 
Let $\tilde{y}$ and $\tilde{x}$ be the output values of the linear programming $I_L$.
Then it is clear that the {\it centdian-distance} of the optimal solution for the fractional $p$DP is a lower bound on the 
{\it centdian-distance} of 
the optimal solution for the $p$DP. Moreover, the linear programming can be solved in polynomial time~\cite{karlp,kha}.

\begin{lemma} Given a solution $\tilde{y}=\{\tilde{y}_1, \tilde{y}_2, \ldots, \tilde{y}_{|V|}\}$
 for the fractional $p$DP, we can determine the
optimal fractional values for $\tilde{x}_{i,j}$.  \label{fkdp}
\end{lemma}

\begin{proof}
Similar with~\cite{Ahn}, for each $i \in V$, we sort $d(i,j)$, $j\in V$, so that 
$d(i,j_1(i)) \le d(i,j_2(i)) \le \ldots \le d(i,j_{|V|}(i))$
and let $s$ be a value such that $\sum_{k=1}^{s-1} \tilde{y}_{j_k(i)}\le 1 \le \sum_{k=1}^{s} \tilde{y}_{j_k(i)}$.
Then let  $\tilde{x}_{i,j}=\tilde{y}_j$ for each $j=j_1(i), j_2(i),\ldots, j_{s-1}(i)$, $\tilde{x}_{i,j_s(i)}=1-\sum_{k=1}^{s-1}\tilde{y}_{j_k(i)}$, and otherwise
$\tilde{x}_{i,j}=0$.
 \end{proof}

Given  a fractional solution $\tilde{x}_{i,j}$, for each $i \in V$, 
let 
$\tilde{D}(i)= \sum_{j \in V} d(i,j) \tilde{x}_{i,j}$ be the distance of assigning vertex
$i$ to its fractional centdian. Given  $\epsilon> 0$, we also let 
the neighborhood set $N(i)$ of vertex $i$  be 
$N(i)=\{j\in V| d(i,j) \le (1+\epsilon) \tilde{D}(i)\}$.

\begin{lemma}~\cite{lin} For each $i \in V$ and $\epsilon > 0$, 
we have  $\sum_{j \in N(i)} \tilde{y}_j \ge \sum_{j \in N(i)} \tilde{x}_{i,j} > \epsilon/(1+\epsilon)$. \label{fkdp1}
\end{lemma}

Then 
we transform the $p$DP to the SCP.  An instance of SCP contains a finite set $\mathcal{U}$ of elements, a collection  $\mathcal{S}$ of (non-empty) subsets of $\mathcal{U}$. 
We let  each vertex $i \in V$ correspond to each element in $\mathcal{U}$, and  
each vertex $j \in V$ with $\tilde{y}_j>0$ correspond to each set in  $\mathcal{S}$, respectively. Then 
for each vertex $i \in V$, if $j \in N(i)$, 
then the corresponding element of $i$ in $\mathcal{U}$ belongs to  the corresponding set of  $j$  in $\mathcal{S}$.
 
Then we use the  greedy approximation algorithm for the SCP  whose approximation ratio is $(\ln |\mathcal{U}|+1)$~\cite{chv,joh, lov} to find a set cover of $\mathcal{U}$ and $\mathcal{S}$. 
Let $A_{SCP}$ be the greedy approximation algorithm for the SCP. Finally, output the corresponding vertex set  for the output set  by $A_{SCP}$. 
Given a graph $G=(V,E, \ell)$, let $P_{APX}$ be a vertex set in $G$. 
Initially, $P_{APX}$ is empty.
Now, for clarification, we describe the $(1+\epsilon)$-approximation  algorithm
for the $p$DP as follows. 

\begin{description} 
\item[\textbf{Algorithm APX-$p$DP}]
\item[\textbf{Input:}]   A  connected, undirected complete graph $G=(V,E,\ell)$ with  a nonnegative length function $\ell$ on edges, an integer $p>0$, and  a real number $\epsilon$, $0<\epsilon<1$.
\item[\textbf{Output:}]  A vertex set $P_{APX}$ with $|P_{APX}| \leq (1+1/\epsilon)(ln|V|+1)p$. 
\item[1.] Let $P_{APX} \leftarrow \emptyset$. 
\item[2.] Use linear programming  ($I_L$) to solve the fractional $p$DP and find  the fractional solutions $\tilde{y}$ and $\tilde{x}$.
\item[3.] For each $i \in V$, compute $\tilde{D}(i)$ and find its neighborhood set $N(i)=\{j\in V| d(i,j) \le (1+\epsilon) \tilde{D}(i)\}$.
\item[4.] {\bf For} each $i \in V$  {\bf do}
\begin{description}
\item [\ ] create an element $u_i$  in $\mathcal{U}$.
\end{description}
 \item[\ \ \ \ \ ] {\bf end for}  
\item[5.] {\bf For} each $j\in V$ with  $\tilde{y}_j>0$ {\bf do}
\begin{description}
\item [\ ] create a subset $\mathcal{S}_j=\{u_i|$ if $j \in N(i)\}$  of $\mathcal{U}$ in $\mathcal{S}$.
\end{description}
 \item[\ \ \ \ \ ] {\bf end for}  
\item[6.] Use  the greedy approximation algorithm $A_{SCP}$ for the SCP to find a  set cover $\mathcal{S}^{\prime}$ of  the instance $\mathcal{U}$ and $\mathcal{S}$. 
Let $y_j=1$ if $\mathcal{S}_j \in \mathcal{S}^{\prime}$, and then $x_{i,j}=1$ if  set $\mathcal{S}_j \in \mathcal{S}^{\prime}$ and $u_i$ is covered by $\mathcal{S}_j$,
and otherwise is $0$.
\item[7.] Let $P_{APX}$ be the corresponding vertex set of $\mathcal{S}^{\prime}$.
\end{description}

The result of this section is summarized in the following theorem.
\begin{theorem} Algorithm APX-$p$DP is a $(1+\epsilon)$-approximation algorithm for the $p$DP
satisfying $|P_{APX}| \le (1+1/\epsilon)(ln|V|+1)p$,  in which $\epsilon>0$. \label{theorem pdp}
\end{theorem}
 
\begin{proof}
Let $P_{OPT}$ be the optimal solution for the $p$DP.
Clearly, 
by Step 5 and Step 6,  a subset $\mathcal{S}_j$ contains the element $u_i$ in $\mathcal{U}$ if $d(i,j) \leq (1+\epsilon)\tilde{D}(i)$, where $i$ is the corresponding vertex of $u_i$ and $j$ is the corresponding vertex of $\mathcal{S}_j$, and each $i \in V$, $\sum_{j \in V}$ $d(i,j) x_{i,j} \leq (1+\epsilon)\tilde{D}(i)$
Hence, we have  
\begin{eqnarray}
\pounds_M(P_{APX})+\pounds_C(P_{APX})
&\le& \sum_{i \in V} \sum_{j \in V} d(i,j) x_{i,j}  
+\max_{i \in V}  \sum_{j \in V} d(i,j) x_{i,j} \nonumber\\
&\le& \sum_{i \in V}  (1+\epsilon)\tilde{D}(i)  
+\max_{i \in V} (1+\epsilon)\tilde{D}(i) \nonumber\\ 
&\le& (1+\epsilon) \pounds_M(P_{OPT}) 
+ (1+\epsilon) \pounds_C(P_{OPT}), \nonumber\\ \nonumber
\end{eqnarray}
since the {\it centdian-distance} of the fractional $p$DP is a lower bound on 
the {\it centdian-distance} of the optimal solution for the $p$DP.

Then we show $|P_{APX}| \le (1+1/\epsilon)(ln|V|+1)p$. By~\cite{lin} and Lemma~\ref{fkdp1}, we have the cardinality  of set  for 
the optimal fractional cover is less than $(1+1/\epsilon)p$ and the cardinality of set by the output of the greedy algorithm is 
at most $(\ln|\mathcal{U}|+1)$~\cite{chv,lov}  of the cardinality  of set for the optimal fractional cover. 
Immediately, we have $|P_{APX}| \le (1+1/\epsilon)(ln|V|+1)p$.
\end{proof}

\section{An Approximation Algorithm for the  CDP}
\label{sectapp2}

In this section, we show a $(1+1/\epsilon)(ln|V|+1)$-approximation algorithm for the CDP  satisfying the {\it centdian-distance}  is less than or equal to $(1+\epsilon)U$, $\epsilon>0$. We only run Algorithm APX-$p$DP for the $p$DP, for $p=1$ to $|V|$ and find the first {\it centdian set} such its {\it centdian-distance} is less than or equal to $(1+\epsilon)U$. 

For the completeness, we describe the approximation algorithm for the CDP and obtain  the {\it centdian set} $P_{\gamma}$ as follows. 

\begin{description}
\item[\textbf{Algorithm APX-CDP}]
\item[\textbf{Input}]  A  connected, undirected complete graph $G=(V,E,\ell)$ with  a nonnegative length function $\ell$ on edges, an 
integer $U>0$  and  a real number $\epsilon$, $0<\epsilon<1$.
\item[\textbf{Output:}]  A vertex set $P_{\gamma}$ with $\pounds_C(P_{\gamma})+\pounds_M(P_{\gamma}) \leq (1+\epsilon)U$.
 \item[1.] Let $p=1$ and  $P_{\gamma} \leftarrow \emptyset$.
\item[2.] Use Algorithm APX-$p$DP to find a vertex set $P_p$ that satisfies Theorem~\ref{theorem pdp}.
\item[3.] 
{\bf If} $\pounds_C(P_{p})+\pounds_M(P_{p})>(1+\epsilon)U$  {\bf then} 
\begin{description}
 Let $p=p+1$ and go to step 2. 
\end{description}
\item[4.]  Let   $P_{\gamma} \leftarrow P_p$.
\end{description}

\begin{theorem} Algorithm APX-CDP is a $(1+1/\epsilon)(ln|V|+1)$-approximation algorithm for the CDP
satisfying the {\it centdian-distance} is less than or equal to $(1+\epsilon)U$,  in which $\epsilon>0$.
\end{theorem}
\begin{proof}

Let $P^{\prime}$ be the {\it centdian set} of  optimal solutions for the CDP with an integer $U$.
We have $\pounds_C(P^{\prime})+\pounds_M(P^{\prime}) \le U$. 
Let  $P^{\prime\prime}$ (respectively, $P^{\gamma}$) be the {\it centdian set} of optimal solutions for the $p$DP 
with $p=|P^{\prime}|$ (respectively, $p=\gamma$). 
Clearly, $\pounds_C(P^{\prime\prime})+\pounds_M(P^{\prime\prime}) \le \pounds_C(P^{\prime})+\pounds_M(P^{\prime}) \le U$.
If $p=|P^{\prime\prime}|$, Algorithm APX-CDP returns a {\it centdian set} $P_{|P^{\prime\prime}|}$ such that 
$\pounds_C(P_{|P^{\prime\prime}|})+\pounds_M(P_{|P^{\prime\prime}|}) \le (1+\epsilon) \pounds_C(P^{\prime\prime})+\pounds_M(P^{\prime\prime}) 
\le (1+\epsilon)U$. 
Since Algorithm APX-CDP returns the  first {\it centdian set} such its {\it centdian-distance} is less than or equal to $(1+\epsilon)U$, 
we have that  $\gamma$ is less than or equal to $|P^{\prime\prime}|$.
By Theorem~\ref{theorem pdp}, we have 
\begin{eqnarray}
|P_{\gamma}| &\le& (1+1/\epsilon)(ln|V|+1) \gamma 
\le (1+1/\epsilon)(ln|V|+1) |P^{\prime\prime}|= (1+1/\epsilon)(ln|V|+1) |P^{\prime}|, \nonumber
\end{eqnarray}
 and
\begin{eqnarray}
\pounds_C(P_{\gamma})+\pounds_M(P_{\gamma}) &\le&  (1+\epsilon) (\pounds_C(P^{\gamma})+\pounds_M(P^{\gamma}))  \nonumber\\ 
&\le&  (1+\epsilon) (\pounds_C(P^{\prime\prime})+\pounds_M(P^{\prime\prime})) \nonumber\\ 
&\le& (1+\epsilon)  (\pounds_C(P^{\prime})+\pounds_M(P^{\prime}))  \nonumber\\ 
&\le& (1+\epsilon)  U.  \nonumber
\end{eqnarray}
\end{proof}

\section{Conclusion}
\label{result}
In this paper, we have investigated the $p$DP  and the CDP and prove that these problems are NP-Complete even when the {\it centdian-distance} is $\pounds_C(V^{\prime})+\pounds_M(V^{\prime})$. 
Then we have presented  non-trivial brute force exact algorithms for the $p$DP and the  CDP, respectively.
Moreover, we have designed a  $(1+\epsilon)$-approximation algorithm for the $p$DP 
satisfying the cardinality of the {\it centdian set} is less than or equal to $(1+1/\epsilon)(ln|V|+1)p$ and a
$(1+1/\epsilon)(ln|V|+1)$-approximation algorithm for the CDP 
satisfying the {\it centdian-distance} is less than or equal to $(1+\epsilon)U$,  in which $\epsilon>0$.
It would be interesting to find approximation complexities for the $p$DP and the CDP. 
Another direction for future research  is whether the $p$DP has a polynomial time exact algorithm for some special graphs.

\end{document}